\documentclass[12pt,a4paper]{amsart}
\usepackage[foot]{amsaddr}
\usepackage{amsmath}
\usepackage{amssymb}
\usepackage{amsthm}
\usepackage{array}
\usepackage[english]{babel}
\usepackage{bigints}
\usepackage{caption}
\usepackage{graphicx}
\usepackage{indentfirst}
\usepackage{mathptmx}
\usepackage{mathtools}
\usepackage{xcolor}
\usepackage[numbers]{natbib}

\newtheorem{definition}{Definition}
\newtheorem{theorem}{Theorem}
\newtheorem{lemma}{Lemma}

\title{Positive Polynomials on closed boxes}
\author{Diniz, M. A.}
\email{marcio.alves.diniz@gmail.com}
\author{Stern, R. B.}
\email{rbstern@gmail.com}
\author{Salasar, L. E.}
\email{luis.salasar@gmail.com}
\address{Universidade Federal de S. Carlos, Dep. of Statistics, Rod. Washington Luis, km 235, S. Carlos - SP, Brazil}
\date{}
\subjclass[2010]{Primary 12D10, 26C05}
\keywords{positive polynomials, unit box, Bernstein polynomials}

\begin{document}

\maketitle

\begin{abstract}

We present two different proofs that positive polynomials on closed boxes of $\mathbb{R}^2$ can be written as bivariate Bernstein polynomials with strictly positive coefficients.
Both strategies can be extended to prove the analogous result for polynomials that are positive on closed boxes of $\mathbb{R}^n$, $n>2$.

\end{abstract}

\section{Introduction}
\label{sec:intro}

The goal of this paper is to show that real polynomials that are strictly positive on closed boxes have a representation with positive coefficients when written using Bernstein's polynomial basis.
More specifically, we will prove the result for the unit box $I=[0,1]\times[0,1]$, i. e. we present new proofs for the following theorem:
\begin{theorem}\label{theo:main}
 If $p: \mathbb{R}^{2} \rightarrow \mathbb{R}$ is such that 
 \begin{equation}\label{eq:pol}
  p(x_1, x_2) = \sum_{i=0}^{n_1} \sum_{j=0}^{n_2} a_{i,j} ~ x_1^i x_2^j
 \end{equation}
 and, for every $(x_1,x_2) \in I$, $p(x_1,x_2) > 0$, 
 then there exist $q_1 \geq n_1, q_2 \geq n_2$ and 
 $C_{i,j} > 0$, $(i,j) \in Q_1\times Q_2$, such that
 \begin{align*}
  p(x_1,x_2) = \sum_{i=0}^{q_1}\sum_{j=0}^{q_2} C_{i,j} ~ x_1^{i}(1-x_1)^{q_1-i}x_2^{j}(1-x_2)^{q_2-j},
 \end{align*}
where $Q_1=\{0,1,\ldots,q_1\}$ and $Q_2=\{0,1,\ldots,q_2\}$.
\end{theorem}
Furthermore, we constructively derive the values of $q_{1}$ and $q_{2}$.

Theorem \ref{theo:main} is an extension of similar results obtained for positive polynomials on compact intervals and multidimensional simplexes by, respectively, Bernstein \cite{bernstein1915}, Hausdorff \cite{hausdorff1921} and P\'olya \cite{polya1928}.
We are aware that, using a different proof strategy, Cassier \cite{cassier1984} has proven a general result from which a similar version of Theorem \ref{theo:main} follows.
We discuss this more extensively at the final section.

We provide two proofs of Theorem \ref{theo:main}.
The first one is supported by results for the univariate version of Theorem \ref{theo:main}, proved by Powers and Reznick \cite{powers2000}.
The second proof extends the approach in Garloff \cite{garloff1986} and Rivlin \cite{rivlin1970}.

The paper is organized as follows.
Section \ref{sec:def} establishes notation and brings the relevant definitions used in the paper.
In Section \ref{sec:aux} we present the auxiliary results.
These results are used in one of the proofs of Theorem \ref{theo:main}, given in Section \ref{sec:main}.
Section \ref{sec:alt} brings an alternative proof, 
based on \cite{garloff1986} and \cite{rivlin1970}.

\section{Definitions and notation}
\label{sec:def}

\begin{definition} \
 Let $\mathcal{P}_{n}$ be the linear space of polynomials of degree $n$, i.e.

\begin{center}
$\mathcal{P}_{n} = \{p: \mathbb{R} \rightarrow \mathbb{R}, 
	       \text{ where } \exists a_{i} \in \mathbb{R}, 0 \leq i \leq n : p(x) = \sum_{i=0}^{n}{a_{i} x^{i}}\}$.	
\end{center}
\end{definition}

\begin{definition}
For any $p \in \mathcal{P}_{n}$ we define its Goursat transform $\tilde{p}$ by

\[\tilde{p}(x) = (2x)^{n}p\left(\frac{1-x}{x}\right).\]
\end{definition}

\begin{definition}
 Let $\mathcal{B}_{n}^{+}$ be the set of polynomials of degree $n$ that can be written with non-negative coordinates in the Bernstein basis,
\end{definition}

\begin{center}
$\mathcal{B}_{n}^{+} = \{p \in \mathcal{P}_{n}, 
	       \text{ where } \exists A_{i} \geq 0:  
				 p(x) = \sum_{i=0}^{n}{A_{i} ~ x^{i}(1-x)^{n-i}}\}.$
\end{center}

Similarly, let $\mathcal{B}_{n}^{+,*}$ be the set of polynomials of degree $n$ that 
can be written with positive coordinates in the Bernstein basis,

\begin{center}
$\mathcal{B}_{n}^{+,*} = \{p \in \mathcal{P}_{n}, 
	       \text{ where } \exists A_{i} > 0: 
				 p(x) = \sum_{i=0}^{n}{A_{i} ~ x^{i}(1-x)^{n-i}}\}.$
\end{center}

\begin{definition}
 \label{def:bernstein_coef}
 For every $a = (a_1, \ldots, a_n) \in \mathbb{R}^{n}$, $m \geq n$ and $0 \leq i \leq m$, let
 \begin{align*}
  A_{i,m}(a) = \sum_{j = 0}^{\min(n, i)}{m - j \choose m - i} a_j.
 \end{align*}
\end{definition}

\begin{definition}
  \label{def:gorsat_coef}
 For every $a = (a_1, \ldots, a_n) \in \mathbb{R}^{n}$, let
 \begin{align*}
  B_{k}(a) = \sum_{(i,j) \in \mathbb{N}^{2}: i-j=n-k}{\left(2^{n}(-1)^{j}{i \choose j}a_{i}\right)}.
 \end{align*}
 Notice that $B_{k}(a)$ is a linear combination of $a$.
\end{definition}

\begin{definition}
 \label{def:fixed_coef}
  For each $0 \leq i \leq n_{1}$ and $0 \leq j \leq n_{2}$,
  let $a_{i,j} \in \mathbb{R}$. For each $0 \leq i \leq n_{1}$,
  define $a_{i}: \mathbb{R} \rightarrow \mathbb{R}$ as
 \begin{align*}
  a_i(x_2) = \sum_{j=0}^{n_2} a_{i, j} \, x_2^{\, j}.
 \end{align*}
  Also define $a(x_{2}) = (a_{0}(x_{2}),\ldots,a_{n_{1}}(x_{2}))$.
\end{definition}

\begin{definition}
 \label{def:fixed_coef1}
 For each $0 \leq i \leq n_{1}$ and $0 \leq j \leq n_{2}$,
 let $a_{i,j} \in \mathbb{R}$. For each $m \geq n_{1}$
 and $0 \leq k \leq n_{2}$, define
 $$b_{k,i,m}(a) = \sum_{j=0}^{\min(n_1,i)}{{m-j \choose m-i}a_{j,k}}.$$
 Also define $b_{i,m}(a) = (b_{0,i,m}(a),\ldots,b_{n_{2},i,m}(a))$.
\end{definition}

\section{Auxiliary results}
\label{sec:aux}


\begin{lemma} \label{lemma_Bernstein}
If $p \in \mathcal{P}_{n}$, $p(x)  = \sum_{i=0}^{n}{a_{i}x^{i}}$, then, for every $m \geq n$, 
 \begin{equation*}
  p(x) = \sum_{i=0}^{m}{A_{i} ~ x^{i}(1-x)^{m-i}}
 \end{equation*}
 if and only if 
 \begin{align}
  \label{eq:polya1}
  A_i = A_{i,m}(a), \quad a = (a_1, \ldots, a_n).
 \end{align}
\end{lemma}
\begin{proof}
Applying the Binomial Theorem to the identity $x^i = x^i (1 - x + x)^{m - i}$, it follows that
$$
x^i = \sum_{j = i}^m {m - i \choose j - i} x^j (1 - x)^{m-j}.
$$
From this expression, we obtain that
$$ p(x) = \sum_{i=0}^{m}{A_{i,m}(a) ~ x^{i}(1-x)^{m-i}}.$$
The proof that the $A_{i}$'s are unique follows from observing that,
 $\{x^{i}(1-x)^{m-i}: 0 \leq i \leq m\}$ is a basis for $\mathcal{P}_{m}$.
\end{proof}

The following theorem is a consequence of Theorem 6 in \cite{powers2000}.

\begin{theorem}
 \label{theo:pos1}
 Let $p \in \mathcal{P}_{n}$ be such that $p(x) > 0$ for all $x \in [0,1]$. Let $\lambda = \min_{x \in [0,1]}p(x)$ and 
 $e_{j}$ be such that $\tilde{p}(x) = \sum_{i=0}^{n}{e_{j}x^{j}}$.
 If $q \geq 3n + \frac{2n^{2}\max_{j}{|e_{j}|}}{\lambda} + 1$,
 then $p \in \mathcal{B}_{q}^{+}$.
\end{theorem}

\begin{lemma}
 \label{lemma:pos2}
 Let $p \in \mathcal{P}_{n}$ be such that $p(x) > 0$ for all $x \in [0,1]$. Let $\lambda = \min_{x \in [0,1]}p(x)$ and 
 $e_{j}$ be such that $\tilde{p}(x) = \sum_{i=0}^{n}{e_{j}x^{j}}$.
 Let $q = 3n + \lceil\frac{2n^{2}\max_{j}{|e_{j}|}}{\lambda}\rceil + 1$, where $\lceil y \rceil = \min \{n \in \mathbb{N}: n \geq y\}$. Then, for every $q^{*} \geq 2q$, 
 $p \in \mathcal{B}_{q^{*}}^{+,*}$.
\end{lemma}

\begin{proof}
It follows from Theorem \ref{theo:pos1} that
there exist $A_{i} \geq 0$ such that
$$p(x) = \sum_{i=0}^{q}{A_{i}x^{i}(1-x)^{q-i}}.$$
Note that
\begin{align*}
  p(x)	&= \sum_{i=0}^{q}{A_{i}x^{i}(1-x)^{q-i}}																																						\\
				&= \sum_{i=0}^{q}{A_{i}x^{i}(1-x)^{q-i}(x + 1-x)^{q^{*}-q}}																														\\
				&= \sum_{i=0}^{q}{A_{i}x^{i}(1-x)^{q-i}\sum_{j=0}^{q^{*}-q}{{q^{*}-q \choose j}x^{j}(1-x)^{q^*-q-j}}}										\\
				&= \sum_{k=0}^{q^{*}}{\left(\sum_{l=\max(0,k+q-q^{*})}^{\min(q,k)}{{q^{*}-q \choose k-l}A_{l}}\right)x^{k}(1-x)^{q^{*}-k}}.
\end{align*}

Observe that, for every $k$, 
$\sum_{l=\max(0,k + q - q^{*})}^{\min(q,k)}{{q^{*}-q \choose k-l}A_{l}} \geq \min(A_{0},A_{q}) > 0$, since $A_0 = p(0) > 0$ and $A_q = p(1) > 0$. Therefore, $p \in \mathcal{B}_{q^{*}}^{+,*}$.
\end{proof}

\begin{lemma}
 \label{lemma:Goursat}
 If $p(x) = \sum_{i=0}^{n}{a_{i}x^{i}}$ and $a = (a_1, \ldots, a_n)$,
 then $$\tilde{p}(x) = \sum_{k=0}^{n}{B_{k}(a)x^{k}}.$$
\end{lemma}

\begin{proof}
 \begin{align*}
  \tilde{p}(x)	&= (2x)^{n} p\left(\frac{1-x}{x}\right)										\\
			    &= (2x)^{n} \sum_{i=0}^{n}{a_{i}\left(\frac{1-x}{x}\right)^{i}}				\\
				&= \sum_{i=0}^{n}{2^{n}a_{i}(1-x)^{i} x^{n-i}}								\\
				&= \sum_{i=0}^{n}{2^{n}a_{i}\sum_{j=0}^{i}{i \choose j}(-1)^{j}x^{n-i+j}}	\\
				&= \sum_{k=0}^{n}{\sum_{(i,j): i-j=n-k}{\left(2^{n}(-1)^{j}{i \choose j}a_{i}\right)}x^{k}} = \sum_{k=0}^{n}{B_{k}(a)x^{k}}.
 \end{align*}
\end{proof}
\section{Proof of Theorem \ref{theo:main}}
\label{sec:main}

The main idea behind this proof is to use twice the positive representation result for univariate polynomials (lemma \ref{lemma:pos2}).
For every fixed value in one of the coordinates of a bivariate polynomial, the function of the free coordinate is a 
univariate polynomial. This polynomial admits a positive Bernstein representation.
Furthermore, the coefficients of this representation are univariate polynomials on the coordinate that was fixed, allowing another application of the positive Bernstein representation theorem for univariate polynomials.
As a result of both applications, a positive Bernstein representation for the bivariate polynomial is obtained.
This strategy can be extended by induction to arbitrary $n$-variate polynomials.



\begin{proof}
 For a given $x_2 \in [0,1]$, obtain from definition \ref{def:fixed_coef} that
 \begin{align*}
  p_{x_2}(x_1) = p(x_1, x_2) = \sum_{i=0}^{n_1} a_i(x_2) x_1^i,
 \end{align*}
Thus, $p_{x_2} \in \mathcal{P}_{n_{1}}$ and $p_{x_2}(x_1) > 0$  for all $x_1 \in [0,1]$.
From this observation, one can obtain two facts. 
First, since $I$ is compact, then $\lambda =  \inf_{(x_1,x_2) \in I^2} p(x_1,x_2) > 0$ and
\begin{equation}
\lambda_{x_2} = \inf_{x_1 \in [0,1]} p_{x_2}(x_1) \geq \lambda > 0 \label{powers_condition_1}. 
\end{equation}

Second, it follows from Lemma \ref{lemma:Goursat} that
\begin{align*}
\tilde{p}_{x_2}(x_1) &:= \sum_{i=0}^{n_{1}}{B_{i}(a(x_2))x_1^{i}}.
\end{align*}
Since each $B_{i}$ is a linear combination of the elements of $a(x_2)$
 and each element of $a(x_2)$ is a polynomial on $x_2$,
 $B_{i}(a(x_2))$ is a polynomial on $x_2$.
 Since $[0,1]$ is compact, there exists $L < \infty$ such that
 \begin{align}
  \label{powers_condition_2}
	\sup_{x_2 \in [0,1]}{\max_{i}{|B_{i}(a(x_2))|}} = L.
 \end{align}
 Therefore, it follows from Lemma \ref{lemma:pos2} 
 and Equations \eqref{powers_condition_1} and \eqref{powers_condition_2} that,
 taking $q_{1} = 2\left(3n_1 + \left\lceil\frac{2n_1^{2}\sup_{x_2 \in [0,1]}{\max_{i}{|B_{i}(a(x_2))|}}}{\inf_{(x_1,x_2) \in I^2} p(x_1,x_2)}\right\rceil + 1\right)$,
one obtains that, for all $x_{2} \in [0,1]$, $p_{x_2} \in \mathcal{B}_{q_{1}}^{+,*}$. 
Therefore, it follows from Lemma \ref{lemma_Bernstein} that,
 for all $x_2 \in [0,1]$,
 \begin{align}
  \label{eqn:recursion_1}
  p(x_1,x_2) = p_{x_2}(x_1)	&= \sum_{i=0}^{q_{1}}{A_{i, q_1}(a(x_2))x_1^{i}(1-x_1)^{q_1-i}}
 \end{align}
 where $A_{i, q_1}(a(x_2)) > 0$. Notice that
 \begin{align*}
  A_{i, q_1}(a(x_2)) &= \sum_{j=0}^{\min(n_1,i)}{{q_1-j \choose q_1-i}a_j(x_2)}	\\
  					 &= \sum_{j=0}^{\min(n_1,i)}{{q_1-j \choose q_1-i}\sum_{k=0}^{n_2}{a_{j,k}x_2^k}}	\\
                     &= \sum_{k=0}^{n_2}{\left(\sum_{j=0}^{\min(n_1,i)}{{q_1-j \choose q_1-i}a_{j,k}}\right)x_2^k} = \sum_{k=0}^{n_2}{b_{k,i,q_1}(a)x_{2}^{k} \in \mathcal{P}_{n_{2}}}
 \end{align*}
 
 It follows from Lemma \ref{lemma:pos2} that,
 taking $q_{2} = 2\left(3n_2 + \max_i\left\lceil\frac{2n_2^{2}\max_{j}{|B_{j}(b_{i,q_{1}}(a))|}}{\inf_{x_2 \in I} A_{i, q_1}(a(x_2))}\right\rceil + 1\right)$,
one obtains that
 \begin{align}
  \label{eqn:recursion_2}
  A_{i, q_1}(a(x_2)) &= \sum_{j=0}^{q_{2}}{C_{i,j} ~ x_2^{j}(1-x_2)^{q_{2}-j}} \text{, } 0 \leq i \leq q_{1}
 \end{align}
 where $C_{i,j} > 0$. By applying Equation \eqref{eqn:recursion_2} to 
 Equation \eqref{eqn:recursion_1}, one obtains
 \begin{align*}
  p(x_1,x_2) &= \sum_{i=0}^{q_{1}}{\sum_{j=0}^{q_{2}}{C_{i,j} ~ x_1^{i}(1-x_1)^{q_1-i}x_2^{j}(1-x_2)^{q_{2}-j}}.}
 \end{align*}
\end{proof}


\section{Alternative proof}
\label{sec:alt}

We consider, as before, the bivariate polynomial $p$ given in \eqref{eq:pol} and $\lambda = \inf_{(x_1,x_2) \in I} p(x_1,x_2)$.
For $q_1, q_2 \geq 1$, let us define the bivariate polynomial

\begin{equation}\label{eq:biv_bern}
b_{k,l}^{(q_1,q_2)}(x_1,x_2) = {q_1\choose k} x_1^k (1-x_1)^{q_1-k} {q_2 \choose l} x_2^l(1-x_2)^{q_2-l},
\end{equation}
where $k \in Q_1$ and $l \in Q_2$. The set of polynomials $\{b_{k,l}^{(q_1,q_2)}(x_1, x_2), k \in Q_1, l \in Q_2\}$ are the Bernstein polynomials of degree $q_1$ and $q_2$ and form a basis for the linear space of all bivariate polynomials of the form $\eqref{eq:pol}$ with $n_1 = q_1$ and $n_2 = q_2$.

\begin{lemma} \label{lemma:xy}
If $i \in Q_1$ and $j \in Q_2$, then
\begin{equation}\label{eq:biv_degree}
x_1^ix_2^j=\sum_{k=0}^{q_1}\sum_{l=0}^{q_2}\frac{{k \choose i}{l \choose j}}{{q_1 \choose i}{q_2 \choose j}}b_{k,l}^{(q_1,q_2)}(x_1,x_2),
\end{equation}
where it is assumed that ${m \choose v} = 0$ for integers $m$ and $v$ such that $m < v$.
\end{lemma}

\begin{proof} 
The result follows by applying the Binomial Theorem to the identity $x_1^ix_2^j=x_1^i(1-x_1+x_1)^{q_1-i}x_2^j(1-x_2+x_2)^{q_2-j}$.
\end{proof}

Henceforth, we shall consider $q_1\geq n_1, q_2\geq n_2$.
Then, it follows from Lemma $\ref{lemma:xy}$ that $p(x_1, x_2)$ given in (\ref{eq:pol}) can be rewritten as
\begin{equation} \label{eq:pol_bern}
p(x_1,x_2) = \sum_{k=0}^{q_1}\sum_{l=0}^{q_2} c^{q_1, q_2}_{k, l} b_{k,l}^{(q_1,q_2)}(x_1,x_2),
\end{equation}
where
\begin{equation} \label{eq:coef_bern}
c^{q_1, q_2}_{k, l} = \sum_{i=0}^{n_1}\sum_{j=0}^{n_2} a_{i,j} \dfrac{ {k \choose i} {l \choose j}}{ {q_1 \choose i} {q_2 \choose j}}.
\end{equation}
The $c_{k,l}^{(q_1,q_2)}$ are the Bernstein coefficients and \eqref{eq:pol_bern} is the Bernstein form of $p(x_1, x_2)$. In the sequel, we denote by
\begin{equation*}
c^{(q_1,q_2)} = \min_{(k,l) \in Q_1 \times Q_2} c_{k,l}^{(q_1,q_2)} \label{eq:bern.coef_min}
\end{equation*}
the smallest Bernstein coefficient of $p(x_1,x_2)$.
\begin{theorem} \label{theo:lambda1}
If $p$ is given by \eqref{eq:pol}, then
\begin{equation}
\lambda - c^{(q_1,q_2)} \geq 0.
\end{equation}
\end{theorem}

\begin{proof}
Since $b_{k,l}^{(q_1,q_2)}(x_1,x_2) \geq 0$ for all $(x_1,x_2) \in I$, then
\begin{align*}
c^{(q_1,q_2)} &= \sum_{k = 0}^{q_1} \sum_{l = 0}^{q_2} c^{(q_1,q_2)} b_{k,l}^{(q_1,q_2)}(x_1,x_2) \\
& \leq \sum_{k = 0}^{q_1} \sum_{l = 0}^{q_2} c^{(q_1,q_2)}_{k, l} b_{k,l}^{(q_1,q_2)}(x_1,x_2) \\
&= p(x_1, x_2),
\end{align*}
for all $(x_1, x_2) \in I$, which implies the assertion.
\end{proof}

\begin{theorem} \label{theo:lambda2}
If $p$ is given by \eqref{eq:pol}, $q_1 \geq n_1$ and $q_2 \geq n_2$, then
\[
\lambda - c^{(q_1,q_2)} \leq \gamma_1 \frac{(q_1-1)}{q_1^2} + \gamma_2 \frac{(q_2-1)}{q_2^2},
\]
\noindent where
\[
\gamma_1= \frac{1}{2} \sum_{i=0}^{n_1}\sum_{j=0}^{n_2} |a_{i,j}| i(i-1), \qquad \gamma_2 = \frac{1}{2} \sum_{i=0}^{n_1}\sum_{j=0}^{n_2} |a_{i,j}| j(j-1).
\]
\end{theorem}

\begin{proof} 
For any real function $f(x_1,x_2)$, define its Bernstein approximation on $I$ by
\begin{equation}\label{eq:bern_appr}
B_{q_1,q_2}(f;x_1,x_2)=\sum_{k=0}^{q_1}\sum_{l=0}^{q_2}f\left(\frac{k}{q_1},\frac{l}{q_2}\right)b_{k,l}^{(q_1,q_2)}(x_1,x_2).
\end{equation}

For $0 \leq i \leq n_1$ and $0 \leq j \leq n_2$, let $\delta^{q_1, q_2}_{k,l}(i,j)$, $(k,l) \in Q_1\times Q_2$, be the Bernstein coefficients of the polynomial $B_{q_1,q_2}(x_1^i x_2^j;x_1,x_2) - x_1^i x_2^j$, i.e.,
\begin{equation}\label{eq:diff}
B_{q_1,q_2}(x_1^i x_2^j; x_1,x_2) - x_1^ix_2^j = \sum_{k=0}^{q_1} \sum_{l=0}^{q_2} \delta^{q_1, q_2}_{k,l}(i,j) b_{k,l}^{(q_1,q_2)}(x_1,x_2).
\end{equation}

Then, from Lemma \ref{lemma:xy} and $(\ref{eq:bern_appr})$ , it follows that
\begin{equation} \label{eq:delta.def}
\delta^{q_1, q_2}_{k,l}(i,j) = \bigg(\frac{k}{q_1}\bigg)^i\bigg(\frac{l}{q_2}\bigg)^j - \dfrac{{k\choose i}{l \choose j}}{{q_1 \choose i}{q_2 \choose j}},
\end{equation}
$k \in Q_1, l \in Q_2$.

For any fixed $0 \leq i \leq n_1$ and $0 \leq j \leq n_2$, we can prove that
\begin{equation} \label{eq:delta.ineq}
0 \leq \delta^{q_1, q_2}_{k,l}(i,j) \leq \bigg( \dfrac{q_1 - 1}{q_1^2} \bigg) \dfrac{i (i - 1)}{2} + \bigg( \dfrac{q_2 - 1}{q_2} \bigg) \dfrac{j (j - 1)}{2},
\end{equation}
for all $k \in Q_1$ and $l \in Q_2$. In order to prove $(\ref{eq:delta.ineq})$, it suffices to show that
\begin{align}
0 \leq \varphi^{q_1}_{k}(i) = \bigg(\frac{k}{q_1}\bigg)^i - \dfrac{{k \choose i}}{{q_1 \choose i}} & \leq \bigg(\dfrac{q_1 - 1}{q_1^2}\bigg) \dfrac{i(i-1)}{2}, \qquad \text{for all } k \in Q_1, \label{eq:delta.ineq1}\\
0 \leq \varphi^{q_2}_{l}(j) = \bigg(\frac{l}{q_2}\bigg)^j - \dfrac{{l \choose j}}{{q_2 \choose j}} & \leq \bigg(\dfrac{q_2 - 1}{q_2^2}\bigg) \dfrac{j(j-1)}{2}, \qquad \text{for all } l \in Q_2. \label{eq:delta.ineq2}
\end{align}

Since (\ref{eq:delta.ineq2}) is essentially the same as (\ref{eq:delta.ineq1}), we only present the proof of (\ref{eq:delta.ineq1}). 
Notice that (\ref{eq:delta.ineq1}) clearly holds for $i = 0$, $i = 1$, $k = 0$ and $k = q_1$. Thus, let us consider $1 \leq k \leq q_1 - 1$ and $i \geq 2$.

If $k < i$, then
\begin{equation*}
0 \leq \varphi^{q_1}_{k}(i) = \bigg( \dfrac{k}{q_1} \bigg)^i \leq \bigg( \dfrac{k}{q_1} \bigg)^2 \leq \bigg( \dfrac{q_1 - 1}{q_1} \bigg) \bigg( \dfrac{i - 1}{q_1}\bigg) \leq \bigg( \dfrac{q_1 - 1}{q_1^2}\bigg) \dfrac{i (i - 1)}{2}.
\end{equation*}

If $k \geq i$, then
\begin{equation*}
\varphi^{q_1}_{k}(i) = \bigg(\frac{k}{q_1}\bigg)^i  - \dfrac{{k \choose i}}{{q_1 \choose i}} = \bigg(\dfrac{k}{q_1} \bigg)^i \bigg[1 - \prod_{r = 0}^{i-1} \dfrac{\big(1 - r/k\big)}{\big(1 - r/q_1\big)} \bigg].
\end{equation*}

Since $0 \leq (1 - r/k) \leq (1 - r/q_1) \leq 1$ for all $r = 0, \ldots, i-1$, it follows that
\begin{equation}
0 \leq \varphi^{q_1}_{k}(i) \leq \bigg(\dfrac{k}{q_1} \bigg)^i \bigg[1 - \prod_{r = 0}^{i-1} \bigg(1 - \frac{r}{k} \bigg) \bigg]. \label{eq:varphi_q1_k}
\end{equation}

Using the fact that, for any $z_1, \ldots, z_m \in [0,1]$, we have
$$
\prod_{i=1}^m (1 - z_i) \geq  1 - \sum_{i = 1}^m z_i,
$$
it follows from (\ref{eq:varphi_q1_k}) that
\begin{equation}
0 \leq \varphi^{q_1}_{k}(i) \leq \bigg(\dfrac{k}{q_1} \bigg)^{i} \dfrac{i(i-1)}{2k} = \bigg(\dfrac{k}{q_1} \bigg)^{i-1} \dfrac{i(i-1)}{2q_1} \leq \bigg( \dfrac{q_1 - 1}{q_1^2} \bigg) \dfrac{i(i-1)}{2},
\end{equation}
which finishes the proof of (\ref{eq:delta.ineq1}) and consequently proves \eqref{eq:delta.ineq}.

Considering the form \eqref{eq:pol} of $p(x_1, x_2)$ and the Bernstein approximation \eqref{eq:bern_appr}, we obtain
\begin{equation*}
B_{q_1, q_2}(p; x_1, x_2) - p(x_1, x_2) = \sum_{i = 0}^{n_1} \sum_{j = 0}^{n_2} a_{i, j} \bigg[B_{q_1, q_2} (x_1^i x_2^j; x_1, x_2) - x_1^i x_2^j \bigg],
\end{equation*}
which implies, using \eqref{eq:diff},
\begin{equation} \label{eq:diff1}
B_{q_1, q_2}(p; x_1, x_2) - p(x_1, x_2) = \sum_{k = 0}^{q_1} \sum_{l = 0}^{q_2} \Bigg( \sum_{i = 0}^{n_1} \sum_{i = 0}^{n_2} a_{i,j} \delta_{k,l}^{q_1, q_2}(i, j) \Bigg) b_{k, l}^{q_1, q_2} (x_1, x_2).
\end{equation}

Now, considering the form \eqref{eq:pol_bern}, we have
\begin{equation} \label{eq:diff2}
B_{q_1, q_2}(p; x_1, x_2) - p(x_1, x_2) = \sum_{k = 0}^{q_1} \sum_{l = 0}^{q_2} \Bigg( p\left(\frac{k}{q_1},\frac{l}{q_2}\right) - c_{k,l}^{q_1, q_2} \Bigg) b_{k, l}^{q_1, q_2} (x_1, x_2).
\end{equation}

Equating the Bernstein coefficients of expressions \eqref{eq:diff1} and \eqref{eq:diff2}, and using \eqref{eq:delta.ineq}, we conclude that
\begin{align*}
p\left(\frac{k}{q_1}, \frac{l}{q_2} \right) &= c^{q_1, q_2}_{k, l} + \sum_{i=0}^{n_1} \sum_{j = 0}^{n_2} a_{i,j} \delta_{k,l}^{q_1, q_2}(i,j) \\
& \leq c^{q_1, q_2}_{k, l} + \sum_{i=0}^{n_1} \sum_{j = 0}^{n_2} |a_{i,j}| \delta_{k,l}^{q_1, q_2}(i,j) \\
& \leq c^{q_1, q_2}_{k, l} + \gamma_1 \frac{(q_1-1)}{q_1^2} + \gamma_2 \frac{(q_2-1)}{q_2^2},
\end{align*}
from which follows the result.
\end{proof}

From Theorems \ref{theo:lambda1} and \ref{theo:lambda2}, it follows that $c^{(q_1,q_2)} \to \lambda$ as $q_1 \to \infty$ and $q_2 \to \infty$ and, therefore, Theorem \ref{theo:main} follows as a corollary.


\section{Concluding remarks}
\label{sec:final}
The representation of polynomials that are positive on the unit interval or any compact subset of $\mathbb{R}^n$ is an important subject with direct applications related to moment problems.
See \cite{lasserre2006} for more on this relation.
The authors searched for the proof of Theorem \ref{theo:main} precisely to prove that the moment problem on the unit square has a solution---i.e. there is a finite representing measure for a sequence of moments---if and only if there is a positive linear functional for all polynomials that are nonnegative on the unit square.
Not being aware of the work of Lasserre \cite{lasserre2006}, where the result similar to the one we wanted to prove is demonstrated, we used the univariate results from Bernstein \cite{bernstein1915} and Hausdorff \cite{hausdorff1921} as a stepping stone to build the proof for the unit square as described in Section \ref{sec:main}.

Once our proof was concluded, we have found references \cite{garloff1986} and \cite{rivlin1970}, which provided a demonstration for a similar result.
Eventually we came across the book by Lasserre \cite{lasserre2006}, where we found a theorem that is similar to Theorem \ref{theo:main}, proved by Cassier \cite{cassier1984}.
We briefly present such result, giving the formulation of \cite{lasserre2006}.
Let $\mathbb{R}[\mathbf{x}]=\mathbb{R}[x_1,\ldots,x_n]$ be the ring of real multivariate polynomials and $\mathbb{K}$ be a basic semi-algebraic set, subset of $\mathbb{R}^n$

\begin{equation}\label{eq:set}
\mathbb{K}:= \{\mathbf{x}\in\mathbb{R}^n: g_j(\mathbf{x})\geq0, j=1,\ldots,m\},
\end{equation}

\noindent
where $g_j(\mathbf{x})\in\mathbb{R}[\mathbf{x}]$, $j=1,\ldots,m$.
Cassier \cite{cassier1984} has proven the following theorem.

\begin{theorem}
Let $g_j(\mathbf{x})\in\mathbb{R}[\mathbf{x}]$ be affine for every $j=1,\ldots,m$ and assume that $\mathbb{K}$, as defined by \eqref{eq:set}, is compact with nonempty interior.
If $f\in\mathbb{R}[\mathbf{x}]$ is strictly positive on $\mathbb{K}$ then

\[
f = \sum_{\alpha\in\mathbb{N}^m}c_{\alpha}g_1^{\alpha_1}\ldots g_m^{\alpha_m},
\]

\noindent
for finitely many nonnegative scalars $(c_{\alpha})$.
\end{theorem}

If $\mathbf{x}=(x_1,x_2)\in\mathbb{R}^2$, $g_1(\mathbf{x})=x_1$, $g_2(\mathbf{x})=1-x_1$, $g_3(\mathbf{x})=x_2$ and $g_4(\mathbf{x})=1-x_2$, then $\mathbb{K}=[0,1]\times[0,1]=I$.
When $f$ is a positive polynomial on $\mathbb{K}$ the theorem applies and there are nonnegative $c_{\alpha}$ such that

\[
f(x_1,x_2)=\sum_{\alpha\in\mathbb{N}^2}c_{\alpha}x_1^{\alpha_1}(1-x_1)^{\alpha_2}x_2^{\alpha_3}(1-x_2)^{\alpha_4}.
\]

The main difference between the above Theorem and Theorem \ref{theo:main} is that the latter constructively derives the positive integers $q_1$ and $q_2$, the degrees of the Bernstein representation.

Both strategies developed in Sections \ref{sec:main} and \ref{sec:alt} can be generalized to prove similar theorems for polynomials that are positive over arbitrary hypercubes. 

\bibliographystyle{plainnat}

\begin{thebibliography}{10}

\bibitem{bernstein1915}
S. Bernstein, {\it Sur la repr\'esentation des polyn\^{o}mes positifs}, Communications de la Soci\'et\'e math\'ematique de Kharkow, ser. 2, 14 (1915), 227--228.

\bibitem{cassier1984}
G. Cassier, {\it Probl\'eme des moments sur un compact de $\mathbb{R}^n$ et repr\'esentation de polyn\^omes \'a plusieurs variables}, J. Funct. Anal., \textbf{58}, 254--266, (1984). 

\bibitem{garloff1986}
J. Garloff, {\it Convergent Bounds for the Range of Multivariate Polynomials}, in {\it Interval Mathematics 1985}, K. Nickel, Ed., Lecture Notes in Computer Science, vol. 212, pp. 37--56, Springer, Berlin, Heidelberg, New York (1986).

\bibitem{hardy1952}
G. H. Hardy, J. E. Littlewood and G. P\'olya, {\it Inequalities}, 2nd ed., Cambridge University Press, 1952. MR {\bf 13}:727e 

\bibitem{hausdorff1921}
F. Hausdorff, {\it Summationsmethoden und Momentfolgen I}, Math. Zeit. 9 (1921), 74--109.

\bibitem{lasserre2006}
J. B. Lasserre, {\it Moments, Positive Polynomials and their applications}, Imperial College Press, 2006. 

\bibitem{polya1928}
G. P\'olya, {\"Uber positive Darstellung von Polynomen Vierteljschr, Naturforsch}. Ges. Z\"urich 73 (1928 141--145, in Collected Papers 2 (1974), MIT Press, 309--313.

\bibitem{powers2000}
V. Powers, and B. Reznick, \emph{Polynomials that are positive on an interval}, Transactions of the American Mathematical Society, \textbf{352}, 4677--4692, (2000).

\bibitem{rivlin1970}
T.~J. Rivlin, \emph{Bound on a polynomial}, J. Res. Nat. Bur. Standards, 74B, 47--57, (1970). 
\end{thebibliography}

\end{document}